\documentclass[12pt]{amsart}
\usepackage{amssymb}
\usepackage{amsthm}
\usepackage{amsmath}
\usepackage[all]{xy}
\usepackage{hyperref}
\usepackage{graphicx}

\theoremstyle{plain}
\newtheorem{thm}[subsection]{Theorem}
\newtheorem{prop}[subsection]{Proposition}
\newtheorem{lemma}[subsection]{Lemma}
\newtheorem{cor}[subsection]{Corollary}
\newtheorem*{thm*}{Theorem}
\newtheorem*{conj*}{Conjecture}

\theoremstyle{remark}
\newtheorem{rem}[subsection]{Remark}
\newtheorem*{rem*}{Remark}
\newtheorem*{ack}{Acknowledgments}

\theoremstyle{definition}
\newtheorem{dfn}[subsection]{Definition}
\newtheorem{ex}[subsection]{Example}
\newtheorem{problem}[subsection]{Problem}

\newcommand{\A}{\ensuremath{\mathcal A}}
\newcommand{\C}{\ensuremath{\mathbb C}}
\newcommand{\CC}{\ensuremath{\mathcal C}}
\newcommand{\Z}{\ensuremath{\mathbb Z}}
\newcommand{\g}{\ensuremath{\Gamma}}
\newcommand{\s}{\ensuremath{\mathfrak s}}
\renewcommand{\L}{\ensuremath{\Lambda}}
\newcommand{\Aut}{\operatorname{Aut}}
\newcommand{\X}{\ensuremath{\mathcal{X}}}

\renewcommand{\to}{\longrightarrow}
\renewcommand{\bar}{\overline}
\renewcommand{\phi}{\varphi}

\begin{document}

\title{On graphic arrangement groups}

\author[D. Cohen]{Daniel C. Cohen}
\address{Department of Mathematics, Louisiana State University, Baton Rouge, Louisiana 70803}
\email{\href{mailto:cohen@math.lsu.edu}{cohen@math.lsu.edu}}
\urladdr{\href{http://www.math.lsu.edu/~cohen/}
{www.math.lsu.edu/\char'176cohen}}

\author[M. Falk]{Michael J. Falk}
\address{Department of Mathematics and Statistics, Northern Arizona University, Flagstaff, Arizona 86011}
\email{\href{mailto:michael.falk@nau.edu}{michael.falk@nau.edu}}
\urladdr{\href{http://www.cefns.nau.edu/~falk/}
{www.cefns.nau.edu/\char'176falk}}
\thanks{M.F.~partially supported by Fulbright Global Scholar Award}


\begin{abstract}
A finite simple graph \g\ determines a quotient $P_\g$ of the pure braid group, called a graphic arrangement group. We analyze homomorphisms of these groups defined by deletion of sets of vertices, using methods developed in prior joint work with R.~Randell. 
We show that, for a $K_4$-free graph \g, a product of deletion maps is injective, embedding $P_\g$ in a product of free groups. Then $P_\g$ is residually free, torsion-free, residually torsion-free nilpotent, and acts properly on a CAT(0) cube complex. We also show $P_\g$ is of homological finiteness type $F_{m-1}$, but not $F_m$, where $m$ is the number of copies of $K_3$ in \g, except in trivial cases. The embedding result is extended to graphs whose 4-cliques share at most one edge, giving an injection of $P_\g$ into the product of pure braid groups corresponding to maximal cliques of \g. 
We give examples showing that this map may inject in more general circumstances. 
We define the graphic braid group $B_\g$ as a natural extension of $P_\g$ by the automorphism group of \g, and extend our homological finiteness result to these groups.
\end{abstract}

\keywords{pure braid group, hyperplane arrangement, graphic arrangement, homological finiteness type, $K_4$-free graph}

\subjclass[2010]{
20F36, 
32S22, 
52C35, 
20E26
}

\maketitle

\section{Introduction}
\label{intro}

For $1 \leq i < j \leq n$, let $H_{ij}$ be the linear hyperplane in $\C^n$ defined by the equation $x_i=x_j$.
 Let \g\ be a simple graph with vertex set $V=\{1,\ldots, n\}$ and edge set $E_\g \subseteq \binom{V}{2}$. The graphic arrangement determined by \g\ is the complex hyperplane arrangement $\A_\g=\{H_{ij} \mid i<j,  \{i,j\} \in E_\g\}$. The associated {\em graphic arrangement group}, or \emph{graphic pure braid group}, is the fundamental group $P_\g=\pi_1(U_\g)$ of the complement $U_\g=\C^n \setminus \bigcup_{H \in \A_\g} H$ of the graphic arrangement $\A_\g$. 

If $E_\g=\binom{V}{2}$, then $\g=K_n$ is the complete graph on $V$, $\A_\g$ is the braid arrangement of rank $n-1$, $U_\g$ is the configuration space of $n$ distinct labelled points in \C, and $P_\g$ is the $n$-strand pure braid group $P_n$. The Artin presentation of $P_n$ has generators $a_{ij}$ corresponding to the full twist of 
strands $i$ and $j$. 
In terms of the arrangement $\A_\g$, $a_{ij}$ is represented by a certain meridian about the hyperplane $H_{ij}$. For a general graph \g, the group $P_\g$ is the quotient of $P_n$ obtained by setting generators $a_{ij}$ equal to 1, for $\{i,j\} \not \in E_\g$. Thus elements of $P_\g$ are represented by pure braids, modulo level-preserving isotopies, in which pairs of strands corresponding to non-edges of \g\ are allowed to cross, or as a group of motions of points in the plane where those pairs of points are allowed to collide. 

The full braid group $B_n$ is an extension of $P_n$ by the symmetric group $S_n$:
\[
1 \to P_n \to B_n \to S_n \to 1.
\]
The surjection $B_n \to S_n$ is the natural one which sends a braid to the induced permutation of its set of endpoints. The automorphism group $\Aut(\g)$ of the graph \g\ is the subgroup of $S_n$ consisting of those permutations that induce permutations of $E_\g$. There is a corresponding subquotient $B_\g$ of $B_n$ which is an extension of $P_\g$ by $\Aut(\g)$:
\[
1 \to P_\g \to B_\g \to \Aut(\g) \to 1.
\]
We call $B_\g$ the {\em graphic full braid group} corresponding to \g; it is isomorphic to the orbifold fundamental group of the quotient of the ``partial configuration space" $U_\g$ by the action of $\Aut(\g)$ - see Section~\ref{full}.

The induced subgraph of \g\ on vertex set $X \subseteq V$  is the graph $\g_X$ with vertices $X$ and edge set $E_\g \cap \binom{X}{2}$. 
Then $\A_{\g_X}$ is the set of hyperplanes of $\A_\g$ containing the coordinate subspace defined by $x_i=0$ for $i \in V \setminus X$, hence is a localization, or flat, of $\A_\g$. 
One has an inclusion $U_\g \hookrightarrow U_{\g_X}$, inducing a homomorphism $\rho_X \colon P_\g \to P_{\g_X}$. If $\X \subseteq 2^V$ is a collection of pairwise incomparable subsets of $V$, denote by $\rho_\X$ the product homomorphism 
\[
\rho_\X=\textstyle{\underset{X \in \X}{\prod}} \, \rho_X \colon P_\g \longrightarrow \underset{X \in \X}{\displaystyle{\prod}} P_{\g_X}.
\]
These homomorphisms were analyzed in a general setting in Sections 3 and 4 of \cite{CFR10}. In this paper we apply the results that paper to deduce consequences concerning residual and homological finiteness properties of graphic pure and full braid groups, for some families of graphs.

A subset $X$ of $V$ is a {\em clique} of \g\ if $\binom{X}{2} \subseteq E_\g$; if $X$ has cardinality $k$ we may say $X$ is a $k$-clique. If $X$ is a clique of \g\ then $P_{\g_X}$ is isomorphic to the pure braid group $P_{|X|}$; since there is no risk of confusion, in this case we write $P_{\g_X}$ as $P_X$. A graph \g\ is {\em $K_4$-free} if \g\ has no cliques of cardinality four. 

A group $G$ is residually free (respectively, residually torsion-free nilpotent) if no nontrivial element lies in the kernel of every homomorphism from $G$ to a nontrivial free (respectively, torsion-free nilpotent) group.  The pure braid group $P_n$ is residually torsion-free nilpotent for every $n$ \cite{FR88,Mar12}, while the full braid group $B_n$ is not residually nilpotent for $n \geq 5$, since the commutator subgroup $[B_n,B_n]$ is perfect for these $n$ \cite{GL69}.  Since $P_4$ is not residually free \cite{CFR11}, neither $P_\g$ nor $B_\g$ is residually free if \g\ is not $K_4$-free.

Let \g\ be a $K_4$-free graph. Let $\X=\X(\g)$ be the family consisting of the 3-cliques and maximal 2-cliques of \g. Our main result is that the corresponding homomorphism $\rho_\X$ is injective. Since the groups $P_X$, $X \in \X$ are products of free groups, it follows that $P_\g$ is residually free, residually torsion-free nilpotent, residually finite, indeed linear, torsion-free, and has solvable word and conjugacy problems. As a subgroup of a right-angled Artin group,  $P_\g$ acts properly and freely (and $B_\g$ acts with finite stabilizers) on a CAT(0) cube complex, properly but not cocompactly. Consequently $P_\g$, and its finite extension $B_\g$, are a-T-menable groups \cite{CCJJV01}. 

Our argument extends to graphs containing 4-cliques, with the family \X\ including maximal $k$-cliques  of \g\ for $k=2,3,4$, 
provided no two 4-cliques intersect in a 3-clique. For these graphs the group $P_\g$ embeds in a product of free groups and copies of $P_4$, implying $P_\g$ is torsion-free, linear, residually finite, residually torsion-free nilpotent, and acts properly and freely on a CAT(0) complex. If \g\ is not $K_4$-free then $P_\g$ is not residually free. We give examples to show $\rho_\X$ may 
still be injective when the intersection condition is not satisfied. The analysis of those examples utilizes semidirect product structures on $P_\g$, arising from cliques of \g, which 
leads to the notion of a {\em graphic discriminantal arrangement} - see 
Section \ref{pure}.

A group $G$ is of type $F_m$ if and only if there is a $K(G,1)$ having the homotopy type of a CW-complex with finite $m$-skeleton \cite{Br82}. A further consequence of our main result is that, for a $K_4$-free graph \g, the group $P_\g$ is of type $F_{m-1}$ and not of type $F_m$, where $m$ is the number of 3-cliques in \g, provided the incidence graph of $E_\g$ with \X\ contains a cycle; if not then $\rho_\X$ is an isomorphism.  If \g\ is not $K_4$-free we are unable to draw conclusions about finiteness type due to the current lack of a precise description of the higher Bieri-Neumann-Strebel-Renz invariants of $P_4$, as noted in \cite{Zar17}.  
The full graphic braid group $B_\g$ contains $P_\g$ as a subgroup of finite index, hence $B_\g$ is of type $F_m$ if and only if $P_\g$ is. 
However, the group $B_\g$ need not be torsion-free nor residually nilpotent - see Section~\ref{full}. 

The paper is structured as follows. In Section \ref{purebraid} we recall the Artin presentations of $B_n$ and $P_n$ and collect some consequences needed for our analysis. Then in Section \ref{retractive} we recall the main results of \cite{CFR10}.  In Section~\ref{pure} we define graphic pure braid groups and establish that families of cliques have the requisite property to apply the method of \cite{CFR10}. Given a clique of \g, we identify a corresponding splitting of $P_\g$ as a semidirect product. In Sections~\ref{triangle} and \ref{finiteness} we prove the homomorphism $\rho_\X$ described above is injective, for graphic pure braid groups associated with $K_4$-free graphs, or with graphs containing 4-cliques with no common 3-cliques, and derive some consequences. We give an example showing that  
$\rho_\X$ may still be injective, even if the aforementioned hypotheses are not satisfied. 
Finally in Section~\ref{full} we define graphic braid groups and extend our finiteness results to those groups.

\section{Braids and pure braids}\label{purebraid}
Let $B_n$ be Artin's braid group on $n$ strands, with generators $\sigma_1, \ldots, \sigma_{n-1}$ satisfying the braid relations:
\[
\begin{array}{cl}
 \sigma_i\sigma_{i+1}\sigma_i = \sigma_{i+1}\sigma_i\sigma_{i+1}     &   \text{for  $1 \leq i \leq n-2$}  \\
 \sigma_i\sigma_j=\sigma_j\sigma_i \hskip11pt     &   \text{for $|i-j|\geq 2$}
\end{array}
\]
As a geometric braid, $\sigma_i$ is represented by a positive half twist of strands $i$ and $i+1$. 
See, for instance, \cite{Bi75} 
as a general reference on braids.

The braid group $B_n$ may also be realized as a subgroup of the automorphism group of the free group $F_n=\langle x_1,\dots,x_n\rangle$ via the Artin representation. In particular, the element of $\Aut(F_n)$ corresponding to the braid generator $\sigma_i$ is given by
\begin{equation} \label{eq:ArtinRep}
x_j \mapsto
\begin{cases}
x_i^{}x_{i+1}^{}x_i^{-1} & \text{if $j=i$,}\\
x_i^{} &\text{if $j=i+1$},\\
x_j^{} & \text{otherwise.}
\end{cases}
\end{equation}

Let $S_n$ be the group of permutations of $[n]=\{1, \ldots, n\}$. The function mapping $\sigma_i$ to the transposition $\s_i=(i,i+1)$ extends to a surjective homomorphism $p \colon B_n \to S_n$, whose kernel is by definition the $n$-strand pure braid group $P_n$. The Artin generators $a_{ij}$ of $P_n$ are given by 
\[
a_{ij}=\sigma_{j-1}^{} \cdots \sigma_{i+1}^{} \sigma_i^2 \sigma_{i+1}^{-1} \cdots \sigma_{j-1}^{-1} \ \text{for} \ 1 \leq i<j \leq n.
\]

Here and throughout the paper, we may write subscripted ordered pairs or sets as strings of letters or digits, without ambiguity, e.g., $a_{ij}$ means $a_{i,j}$. As a geometric braid, $a_{ij}$ is represented by a positive full twist 
of strands $i$ and $j$  (behind the intermediate strands). 
In terms of the braid arrangement, $a_{ij}$ is represented by a meridian loop about the hyperplane $H_{ij}$.

The pure braid group $P_n=F_{n-1} \rtimes \cdots \rtimes F_2 \rtimes F_1$ admits the structure of an iterated semidirect product of free groups. For $r<s$, the action of $P_r$ on $F_{s-1}=\langle x_1,\dots,x_{s-1}\rangle=\langle a_{1,s},\dots,a_{s-1,s}\rangle$ is given by the restriction of the Artin representation \eqref{eq:ArtinRep}. The iterated semidirect product structure is in evidence in the standard presentation of $P_n$, with generators $a_{i,j}$ as above, and relations
\begin{equation} \label{eq:PureBraidRels}
a_{rs}^{-1}a_{ij}a_{rs}^{}=
\begin{cases}
a_{ij}^{} & \text{if $r<s<i<j$,}\\
a_{rj}^{}a_{ij}^{}a_{rj}^{-1} 
& \text{if $r<s=i<j$,}\\
a_{rj}^{}a_{sj}^{}a_{rj}^{-1}a_{sj}^{-1} a_{ij}^{} a_{sj}^{}a_{rj}^{}a_{sj}^{-1}a_{rj}^{-1} 
& \text{if $r<i<s<j$,}\\
a_{rj}^{}a_{sj}^{}a_{ij}^{}a_{sj}^{-1}a_{rj}^{-1} 
& \text{if $r=i<s<j$,}\\
a_{ij}^{} & \text{if $i<r<s<j$.}
\end{cases}
\end{equation}

For each $m<n$, forgetting the last $n-m$ strands of an $n$-strand pure braid gives rise to a split extension
\begin{equation} \label{eq:Pnm}
1\to P_{n,m} \to P_n \to P_m \to 1.
\end{equation}
The group $P_{n,m}=F_{n-1} \rtimes \cdots \rtimes F_m$, generated by $a_{i,j}$ with $m<j$, may be realized as the fundamental group of the complement of the discriminantal arrangement  $\A_{n,m}$ (in the sense of \cite{SV91}). This arrangement may be realized with the (affine) hyperplanes 
\[
\{x_j = i \mid 1\le i \le m<j \le n\} \cup \{x_j=x_i \mid m<i<j\le n\}
\] 
in $\C^{n-m} \cong \{(1,\dots,m,x_{m+1},\dots,x_n)\} \subset \C^n$. 
The complement is the configuration space of $n-m$ distinct labelled points in $\C\setminus\{m\ \text{points}\}$. The split extension \eqref{eq:Pnm} may be obtained from the long exact homotopy sequence of the Fadell-Neuwirth bundle of configuration spaces 
\begin{equation*} \label{eq:FN}
U_{K_n} \to U_{K_m},\quad (x_1,\dots,x_n) \mapsto (x_1,\dots,x_m),
\end{equation*}
with fiber the complement of $\A_{n,m}$ in $\C^{n-m}$, which admits a section \cite{FN62}.

The presentation of $P_n$ with relations \eqref{eq:PureBraidRels} above is equivalent to the more compact ``Artin presentation'', as in \cite{MarMcC09}. 
The generating set is $\{a_{ij} \mid 1 \leq i < j \leq n\}$ and the 
relations arise from 3- and 4-element subsets of $V=\{1, \ldots, n\}$ as follows:
\begin{eqnarray}\label{commrelns}
{\textrm{(i)}} & [a_{ij},a_{rs}]=1 & \text{if $i<r<s<j$ or $r<s<i<j$};\nonumber \\
{\textrm{(ii)}} & [a_{ij}, a_{sj}a_{rs}a_{sj}^{-1}]=1 & \text{if if $r<i<s<j$};\\
{\textrm{(iii)}} & [a_{ij}a_{ir},a_{jr}]=[a_{ij},a_{ir}a_{jr}]=1 & \text{if $i<j<r$}.\nonumber 
\end{eqnarray}
When the conditions of (i) hold, or the same hold with $(i,j)$ and $(r,s)$ interchanged, we call the partition $ij|rs$  {\em non-crossing}; otherwise $ij|rs$ is {\em crossing}, and relation (ii) applies, again possibly with $(i,j)$ and $(r,s)$ interchanged.

\begin{rem} \label{rem:equiv pure braid rels}
For $r<s<j$, the pure braid relations (iii) imply that, for instance,  
$a_{sj}^{}a_{rs}^{}a_{sj}^{-1} = a_{rj}^{-1}a_{rs}^{}a_{rj}^{}$, since $a_{rs}^{}a_{rj}^{}a_{sj}^{}=a_{rj}^{}a_{sj}^{}a_{rs}^{}$. Using this and related observations, for $r<i<s<j$, (conjugates of) the pure braid relation (ii) may be expressed in the following equivalent ways:
\[
\begin{matrix}
[a_{ij}^{}, a_{ri}^{}a_{rs}^{}a_{ri}^{-1}]=1, & 
[a_{rs}^{}, a_{is}^{}a_{ij}^{}a_{is}^{-1}]=1, & 
[a_{rs}^{}, a_{rj}^{}a_{ij}^{}a_{rj}^{-1}]=1.
\end{matrix}
\]
\end{rem}

By definition, $P_n$ is a normal subgroup of $B_n$. For the purposes of Section~\ref{full} we will also need a particular description of the action of the braid generators $\sigma_i$ on $P_n$. 
\begin{prop}[\cite{CS97,DG81}] \label{conj}The conjugation action of $B_n$ on $P_n$ is given by
\[
\sigma_k^{} a_{i,j}^{}\sigma_k^{-1} = 
\begin{cases}
a_{i-1,i}^{}a_{i-1,j}^{}a_{i-1,i}^{-1}&\text{if $k=i-1$,}\\
a_{i+1,j}^{}&\text{if $k=i<j-1$,}\\
a_{j-1,j}^{}a_{i,j-1}^{}a_{j-1,j}^{-1}&\text{if $k=j-1>i$,}\\
a_{j,j+1}^{}&\text{if $k=j$,}\\
a_{i,j}&\text{otherwise,}
\end{cases}
\]
\end{prop}

Finally, for the purposes of Section~\ref{finiteness}, we recall the fact \cite{Bi75} that the center $Z(P_n)$ of the pure braid group is infinite cyclic, generated by the  
full twist 
\[
\Delta^2=a_{12}a_{13}a_{23}a_{14}a_{24}a_{34}\cdots\cdots a_{1,n} \cdots a_{n-1,n}.
\] 
The quotient $P_n/Z(P_n)$ is isomorphic to the fundamental group of the complement of the projectivization of the 
rank $n-1$ 
braid arrangement.

\section{Retractive families}\label{retractive}
Let $G$ be a group with finite generating set $Y$. For $S \subseteq Y$ denote by $G_S$ the quotient of $G$ by the normal closure of the complement $Y\setminus S$. That is, $G_S$ is obtained from a presentation of $G$ by adding the relations $y=1$ for generators $y \in Y$ lying outside of $S$. Let $\rho_S \colon G \to G_S$ denote the canonical quotient map. If $\X \subseteq 2^Y$, let
\[
\rho_\X =\textstyle{\underset{S \in \X}{\prod}} \, \rho_S \colon G \to \underset{S \in \X}{\displaystyle{\prod}} G_S.
\]

\begin{dfn} A subset $S \subseteq Y$ is {\em retractive} if the composite 
\[
\langle S \rangle \hookrightarrow G \overset{\rho_S}{\longrightarrow} G_S
\]
is injective. A {\em retractive family} is a family $\X \subseteq 2^Y$ of pairwise incomparable nonempty sets with the property that every nonempty intersection of subsets of \X\ and every singleton $\{y\}$ for $y \in \bigcup \X$ is retractive.
\end{dfn}

If the abelianization $G/[G,G]$ is free abelian of rank $|Y|$, then the singleton $\{y\}$ is a retractive set for every $y \in Y$. This will be the case in our setting. We will also use the following criterion. Let $F(Y)$ denote the free group on the set $Y$. If $S \subseteq Y$ we consider $F(S)$ to be a subgroup of $F(Y)$. Let $\pi_S \colon F(Y) \to F(S)$ be defined by 
\[
\phi_S(y)=\begin{cases}y & \text{if} \ y \in S;\\ 1 & \text{if} \ y \not \in S.\end{cases}
\]

\begin{prop}\label{conjfree} Suppose $G$ has presentation $\langle Y \mid R \rangle$, $R \subseteq F(Y)$, and $S \subseteq Y$. Let $R_S=\phi_S(R)$. If $R_S \subseteq R$ then $S$ is retractive.
\end{prop}

\begin{proof} The hypothesis implies that the identity map $S \to S$ induces a well-defined homomorphism $G_S \to \langle S \rangle$ inverse to the quotient map $\langle S \rangle \to G_S$. 
\end{proof}

\begin{cor} \label{cor:semidirect}
Let $G= G_2 \rtimes G_1$ be a semidirect product of groups, with $G_1$ acting on $G_2$, and let $Y=Y_1 \cup Y_2$ be a generating set for $G$ with $Y_i \subseteq G_i$ for $i=1,2$. Then $Y_1$ is retractive.
\end{cor}

If $S \subseteq Y$ is a retractive set, we tacitly identify $G_S$ with the subgroup $\langle S \rangle$ of $G$. If $\X \subseteq 2^Y$ is a family of pairwise incomparable sets, we say a subset of $Y$ is {\em transverse} to \X\ if it is not a subset of any element of \X. We will apply the following result from \cite[Section 3]{CFR10}.

\begin{thm}\label{inj} Suppose $G$ is a group with generating set $Y$ and $\X \subseteq 2^Y$ is a retractive family with $Y=\bigcup \X$. Then the homomorphism
\[
\rho_\X \colon G \to \underset{S \in \X}{\displaystyle{\prod}} G_S
\]
is injective if 
\begin{enumerate}
\item $[a,b]=1$ for all $a, b \in Y$ with $\{a,b\}$ transverse to \X, and
\item $[[G_S,G_S],y]=1$ for all $S \in \X$ and $y \in Y\setminus S$.
\end{enumerate}
\end{thm}

\section{Graphic pure braid groups}\label{pure}
Let \g\ be a graph with vertex set $V=[n]=\{1, \ldots, n\}$ and edge set $E_\g \subseteq \binom{V}{2}$. Let  $\bar{E}_\g=\binom{V}{2} \setminus E_\g$ denote the set of edges of the complete graph $K_n$ not in \g.
Let $\langle \langle S \rangle\rangle$ be the normal closure of a subset $S$ of a group.

\begin{dfn} The \emph{graphic arrangement group} or 
{\em graphic pure braid group} associated with \g\ is the quotient
\[
P_\g=P_n/\langle\langle a_{ij}  \mid \{i,j\} \in \bar{E}_\g \rangle\rangle.
\]  
\end{dfn}

Using the Van Kampen theorem one can show that $P_\g$ is isomorphic to the arrangement group $\pi_1(U_\g)$, where $U_\g = \C^n \setminus \bigcup_{H \in \A_\g} H$ as in the Introduction.

To simplify notation, we will denote the image of $a_{ij}$ in quotients of $P_n$ also by $a_{ij}$, with the meaning made clear from the context.  We will apply the results of Section~\ref{retractive} to $P_\g$ with its generating set $Y=\{a_{ij} \mid \{i,j\} \in E_\g\}$.  

The graphic pure braid group $P_\g$ admits semidirect product structure analogous to that of the pure braid group exhibited in \eqref{eq:Pnm}, arising from a clique in \g. Recall the subgroup $P_{n,m} =\ker(P_n \to P_m)$ from Section~\ref{purebraid}.

\begin{thm} \label{thm:Pg semi}
Let $X=[m]$ be a clique in \g, and let $N=\langle\langle a_{ij}  \mid \{i,j\} \in \bar{E}_\g \rangle\rangle$ in $P_n$. Then there is a split extension
\[
1\to P_{n,m}/N \to P_\g \to P_m \to 1.
\]
\end{thm}
\begin{proof}
Note that, for $\{i,j\}\in \bar{E}_\g$ with $i<j$, we have $m<j$. Then $a_{ij} \in \ker(P_n \to P_m)=P_{n,m}$. It follows that $N$ is, in fact, a (normal) subgroup of $P_{n,m}$. 
So by the third isomorphism theorem, we have $(P_n/N)/(P_{n,m}/N) \cong P_n/P_{n,m} \cong P_m$, yielding a short exact sequence as asserted. One observes that the splitting of the pure braid sequence \eqref{eq:Pnm} descends to a splitting of the sequence 
$1\to P_{n,m}/N \to P_\g \to P_m \to 1$, completing the proof.
\end{proof}

\begin{rem} \label{rem:deleted disc}
The group $P_{n,m}^\g=P_{n,m}/N$ is also an arrangement group, isomorphic to the fundamental group of the complement of the arrangement 
obtained from the discriminantal arrangement $\A_{n,m}$ by removing the hyperplanes $x_j=i$ for $1 \leq i \leq m<j\le n$, and $x_j=x_i$ for $m<i<j\leq n$, when $\{i,j\} \in \bar{E}_\g$. This is the {\em graphic discriminantal arrangement} $\A_{n,m}^\g$, associated to a graph \g\ on $n$ vertices with a given $m$-clique. 
\end{rem}

As shown in \cite[Lemma 4.2]{Ku93}, the edge sets of cliques of \g\ are precisely the connected modular flats of the cycle matroid of \g. So the clique $X$ of Theorem~\ref{thm:Pg semi} corresponds to a modular flat of the graphic arrangement $\A_\g$. The split exact sequence of this theorem arises from the linear fibering of arrangement complements associated with a modular flat \cite{Par00}. 
Combined with \cite[Section 4]{FaPr02}, we obtain the following generalization. A graph \g\ is the {\em generalized parallel connection} of its induced subgraphs $\g_1$ and $\g_2$ along $X$ if $E_\g=E_{\g_1} \cup E_{\g_2}$  and $E_{\g_1} \cap E_{\g_2}=\binom{X}{2}$. In particular, $X$ is a clique, and hence a modular flat in $\g_1$, $\g_2$, and \g. Write $\g = \g_1 \cup_X \g_2$. For $i=1,2$, denote the vertex set of $\g_i$ by $V_i$ and let $n_i=|V_i|$.

\begin{prop} \label{prop:slf}
Suppose $\g= \g_1 \cup_X \g_2$, with $X=V_1 \cap V_2$ an $m$-clique. Let $\X=\{V_1,V_2\}$. 
Then 
\begin{enumerate}
\item there is a pullback diagram of groups
 \[
\xymatrix{
P_\g \ar[r]^-{\rho_{V_1}} \ar[d]_{\rho_{V_2}} & P_{\g_1}\ar[d]^{\rho_X}\\
P_{\g_2} \ar[r]^-{\rho_{X}} &  P_X;
}
\]
\item there is a commuting diagram of split short exact sequences
\[
\xymatrix{
1 \ar[r] & P_{n_1,m}^{\g_1} \ar[r] \ar[d]_{\mathrm{id}} & P_\g \ar[r]^{\rho_{V_2}} \ar[d]^{\rho_{V_1}} & P_{\g_2}  \ar[r] \ar[d]^{\rho_X} & 1 \\
1 \ar[r] & P_{n_1,m}^{\g_1} \ar[r]  & P_{\g_1} \ar[r]^{\rho_X} & P_X  \ar[r] & 1; 
}
\]
\item the map $\rho_\X \colon P_\g \to P_{\g_1} \times P_{\g_2}$ is injective.
\end{enumerate}
\end{prop}

\begin{proof} In \cite{Par00} it is shown that the inclusion $U_{\g_1} \to U_X$, with $X$ a modular flat in $\g_1$, is equivalent to a locally trivial fibration with a section. In \cite{FaPr02}, the notion of generalized parallel connection of arrangements is defined, and shown to correspond to a pullback diagram of modular fibrations. The graphic arrangement of a generalized parallel connection of graphs is the generalized parallel connection of corresponding graphic arrangements, so the pullback diagram applies to the graphic arrangement complements $U_\g$, $U_{\g_i}$, and $U_X$, inducing the pullback diagram in (i). 

The lower row in (ii) is the exact sequence of Proposition~\ref{thm:Pg semi}. The exactness of the upper row, and commutativity of the diagram, follow from the fact that the pullback diagram in (i) is induced by a pullback diagram of locally trivial fibrations. Statement (iii) follows from the definition of pullback diagram.
\end{proof}

In the special case $\g_1=K_n$, $n > m$, Proposition~\ref{prop:slf} appears in 
\cite{CFR10}. 
Without loss, take $X=[m]$. Since $X$ 
is clique in $\g_2$, the set of linear functions $\{x_1,x_2,\dots,x_m\}$ is a generating set for the arrangement $\A_{\g_2}$ in $\C^{n_2}$ in the sense of \cite[Section 2]{CFR10}: all differences $x_i-x_j$, $1\le i<j\le m$, are nowhere zero on the complement $U_{\g_2}$. Consequently, the map $f\colon U_{\g_2} \to \C^m$, $f(x_1,\dots,x_{n_2})=(x_1,\dots,x_m)$, takes values in  the configuration space $U_{K_m}$. As in \cite[Thm.~2.1.3]{CFR10}, one can then check that $U_\g$ is homeomorphic to the total space of the pullback of the configuration space bundle $U_{\g_1}=U_{K_n} \to U_{K_m}$  along the map $f$, yielding an equivalence of bundles between $U_\g \to U_{\g_2}$ and this pullback. The map $f$ is linearly equivalent to $\rho_X$.

The case $\g_1=K_{m+1}$ appears in \cite[Thm.~1.1.5]{C01}, and will be 
particularly useful. 
Since $P_{m+1,m}^{K_{m+1}}= F_m$, we have the following.
\begin{cor}\label{cor:neighbor}
Suppose $\g= K_{m+1} \cup_{[m]} \g'$. Then
\begin{enumerate}
\item there is a commuting diagram of split short exact sequences
\[
\xymatrix{
1 \ar[r] & F_m \ar[r] \ar[d]_{\mathrm{id}} & P_\g \ar[r] \ar[d]^{\rho_{[m+1]}} & P_{\g'}  \ar[r] \ar[d]^{\rho_{[m]}} & 1 \\
1 \ar[r] & F_m \ar[r]  & P_{m+1} \ar[r] & P_{m}  \ar[r] & 1;
}
\]
\item $P_\g$ embeds in $P_{\g'}\times P_{m+1}$;  
\item $P_\g$ embeds in a product of pure braid groups if and only if $P_{\g'}$ does. 
\end{enumerate}

\end{cor}
\begin{rem}
By (iii) above, if \g\ is an iterated generalized parallel connection of complete graphs, then $P_\g$ embeds in a product of pure braid groups. These are precisely the chordal graphs, and the corresponding graphic arrangements are fiber-type \cite[Prop.~2.8]{Sta72}. With this observation, the embedding of $P_\g$ in a product of pure braid groups follows from \cite{C01}.
\end{rem}

We now identify some retractive subsets of $Y=\{a_{ij} \mid \{i,j\} \in E_\g\}$. 
To simplify statements, we will say a subset $X$ of $V$ is retractive if the set $\{{a}_{ij} \mid i,j \in X\}$ is a retractive subset of $Y$. 

\begin{prop} \label{modular} Let $X \subseteq V$ be a clique. Then $X$ is retractive.
\end{prop}
\begin{proof} 
This follows from Theorem \ref{thm:Pg semi} and Corollary \ref{cor:semidirect}.

Alternatively, one can use the Artin presentation to verify the hypothesis of Proposition~\ref{conjfree} directly as follows. Let $K=\{i,j,r,s\} \in \binom{V}{4}$. Suppose $ij|rs$ is non-crossing. If $K \subseteq X$ then $\phi_X([a_{ij},a_{rs}])=[a_{ij},a_{rs}]$. If $K \not \subseteq X$, then $\phi_X([a_{ij},a_{rs}])=1$ 
in $F(X)$, 
 i.e., the relation $[a_{ij},a_{rs}]=1$ in $P_n$ becomes $1=1$ in $P_X$.  
Suppose $ij|rs$ is crossing. 
If $K \subseteq X$ then, because $X$ is a clique, $\phi_X([a_{ij}, a_{sj}a_{rs}a_{sj}^{-1}])=[a_{ij},a_{sj}a_{rs}a_{sj}^{-1}]$. Suppose $K \not \subseteq X$. Then either $\{i,j\}$ or $\{r,s\}$ does not lie in $\binom{X}{2}$, hence $\phi_X([a_{ij}, a_{sj}a_{rs}a_{sj}^{-1}])=1$. 

Finally, suppose $K=\{i,j,r\} \in \binom{V}{3}$. If $K \in \binom{X}{3}$ then the relations \eqref{commrelns}\,(iii) are preserved by $\phi_X$, and if $\{i,j,r\} \not \in \binom{X}{3}$, then at least two of $\{i,j\}$ $\{i,r\}$, and $\{j,r\}$ lie in $\bar{E}_\g$, hence those relations reduce to $1=1$ under $\phi_X$. Then Proposition~\ref{conjfree} implies $X$ is retractive.
\end{proof}

\begin{prop}\label{cliqfam} If $\X$ is a family of cliques in $\g$, then $\X$ is retractive.
\end{prop}
\begin{proof} Since each edge $\{i,j\} \in E_\g$ is a clique, and the intersection of any family of cliques of \g\ is a clique of \g, the statement follows from Proposition~\ref{modular}.
\end{proof}

Recall the pure braid group indexed by strands in $X \subseteq V$ is denoted by $P_X$. A family $\X$ of cliques in $\g$ gives rise to the product homomorphism
\[
 \rho_\X=\textstyle{\underset{X \in \X}{\prod}} \, \rho_X \colon P_\g \longrightarrow \underset{X \in \X}{\displaystyle{\prod}} P_{X}
\]
from the graphic pure braid group to a direct product of pure braid groups.

\begin{ex} \label{ex:brunnian}
If $\X$ is any family of proper cliques in the complete graph $\g=K_n$, the kernel of $\rho_\X$ contains the (nontrivial) group of $n$-strand Brunnian braids, hence $\rho_\X$ is not injective.
\end{ex}

In light of the preceding example, we will let $\X(\g)$ denote the set of \emph{maximal} cliques in the simple graph $\g$.

\begin{ex} \label{ex:deletedK5}
Let \g\ be the 
graph on $V=[5]$ with $E_\g=\binom{V}{2} \setminus \left\{\{4,5\}\right\}$. 
Note that \g\ is chordal. 
Then $\X(\g)$ is the family $\left\{\{1,2,3,4\},\{1,2,3,5\}\right\}$ of two maximal cliques in $\g$. Deletion of the fifth vertex defines the split surjection $P_\g \to P_{1234}$ of Theorem~\ref{thm:Pg semi}, which is the pullback of the deletion map $P_{1234} \to P_{123}$ along the deletion map $P_{1235} \to P_{123}$ - see Proposition~\ref{prop:slf}. 
\[
\xymatrix{
P_\g \ar[r] \ar[d] & P_{1234}\ar[d]\\
P_{1235} \ar[r] &  P_{123}.
}
\]
Since this is a pullback diagram,  
the map $\rho_{\X(\g)}=(\rho_{1234},\rho_{1235}) \colon P_\g \to P_{1234} \times P_{1235}$ is injective. 
\end{ex}

\begin{rem} \label{rem:limit}
In general, the group $P_\g$ may be realized as the limit of a diagram of groups on the nerve of the family $\X(\g)$, with vertex groups $P_X$ and split surjections $P_X \to P_{X'}$ for $X \supseteq X'$. 
See \cite{L-FS} for analogous observations concerning the holonomy Lie algebra of the graphic arrangement $\A_\g$.
\end{rem}

Based on Proposition \ref{prop:slf}, Corollary \ref{cor:neighbor}, and Example \ref{ex:deletedK5}, one might speculate that the map $\rho_{\X(\g)}$ is injective for all graphs \g.  In the next section, we find conditions on the graph \g\ which ensure that $\rho_{\X(\g)}$ is injective

\section{\texorpdfstring{$K_4$}{K4}-free graphs}
\label{triangle}
With the results of the preceding section we are able to apply Theorem~\ref{inj} to $P_\g$ under some conditions. 
There is no loss in assuming \g\ has no isolated vertices, so $|X| \geq 2$ for $X \in \X(\g)$. By Proposition~\ref{cliqfam} we have the following.

\begin{prop} The family $\X(\g)$ is retractive.
\end{prop}

Thus the hypotheses of Theorem~\ref{inj} hold in our setting. It remains to verify conditions (i) and (ii) of that theorem. To simplify terminology, we will say a set of elements of $Y$ is transverse to $\X(\g)$ when it is transverse to the corresponding family in $2^Y$. The following result is originally due to D.~Malcolm \cite{Mal15}.

\begin{prop} \label{comm} Suppose $S=\{a_{ij},a_{rs}\} \subseteq Y$ is transverse to $\X(\g)$. Then 
\[
[a_{ij},a_{rs}]=1
\]
in $P_\g$.
\end{prop}

\begin{proof} First suppose $|\{i,j,r,s\}|=3$. Since $S$ is transverse to $Y$, $S$ is not a subset of any 3-clique of \g, hence one of the three elements of $\binom{S}{2}$ lies in $\bar{E}_\g$ and the corresponding pure braid generator maps to $1$ in $P_\g$. Then the pure braid relation \eqref{commrelns}\,(iii) implies that the other two generators commute in $P_\g$, that is, $[a_{ij},a_{rs}]=1$ in $P_\g$. We leave the case-by-case verification to the reader.

Suppose $\{i,j,r,s\}=4$. If $ij|rs$ is non-crossing, then $[a_{ij},a_{rs}]=1$ in $P_n$, hence in $P_\g$, by \eqref{commrelns}\,(i). Suppose $ij|rs$ is crossing; without loss of generality, $r<i<s<j$. Since $S$ is not a 4-clique of \g, one of the four elements of 
\(
\binom{S}{2} \setminus \{\{i,j\},\{r,s\}\}
\) 
lies in $\bar{E}_\g$, so the corresponding pure braid generator maps to 1 in $P_\g$. Then, using \eqref{commrelns}\,(ii) or one of the consequent relations listed in Remark~\ref{rem:equiv pure braid rels}, we conclude $[a_{ij},a_{rs}]=1$ in $P_\g$.
\end{proof}

We will verify condition (ii) of Theorem \ref{inj} using the following.

\begin{lemma} Let $X$ be a 3-clique in a graph \g, and let $g \in P_\g$. Suppose  $[g,a_{ij}]=1$ for two of the three elements $\{i,j\} \in \binom{S}{2}$. Then $[[P_X,P_X],g]=1$.
\label{twoof}
\end{lemma}
\begin{proof} By \eqref{commrelns}\,(iii), $P_X$ has presentation $\langle a,b,c \mid abc=bca=cab \rangle$. Let $x=abc$. Then $P_X = \langle x,y,z \mid [x,y]=[x,z]=1\rangle$, with $y$ and $z$ equal to any two of $a,b,c$. Setting $F=\langle y,z \rangle$, one has $P_X \cong \langle x \rangle \times F$, so $[P_X,P_X]=[F,F]$. Since $y$ and $z$ may be chosen so that $g$ centralizes $F$, it follows that  $[[P_X,P_X],g]=1$.
\end{proof}

\begin{prop}\label{triple} Suppose \g\ is $K_4$-free.  Let $X \in \X(\g)$ and $\{r,s\} \not \subseteq X$. Then 
\[
[[P_X,P_X],a_{rs}]=1
\]
in $P_\g$.
\end{prop}

\begin{proof} By hypothesis, the set $\X(\g)$ of maximal cliques in \g\ consists of 2- and 3-cliques. We may assume $X$ is a 3-clique in \g. By Lemma~\ref{twoof} and Proposition~\ref{comm}, $[[P_X,P_X],a_{rs}]=1$ in $P_\g$ if $a_{rs}$ together with any two of the three generators of $P_X$ forms a set transverse to $\X(\g)$. This must be the case, since $X \cup \{r,s\}$ has cardinality at least 4.
\end{proof}

\begin{thm}\label{k4free} Suppose \g\ is $K_4$-free. Then
\[
\rho_{\X(\g)} \colon P_\g \to \underset{X \in \X(\g)}{\displaystyle{\prod}} P_X
\]
is injective.
\end{thm}

\begin{proof} This is now a consequence of Theorem~\ref{inj}, using Propositions~\ref{modular}, \ref{comm}, and \ref{triple}.
\end{proof}

\begin{cor}\label{props} Suppose \g\ is a $K_4$-free graph. 
Then $P_\g$ is residually free, torsion-free, residually torsion-free nilpotent, linear, and residually finite. Moreover $P_\g$ acts freely and properly on a CAT(0) cube complex.
\end{cor}

\begin{proof} For each $X \in \X(\g)$, $P_X$ is isomorphic to $F_2 \times \Z$ if $|X|=3$, or to \Z\ if $|X|=2$. Then the properties in the first list hold for each $P_X$, and are preserved under direct products and inherited by subgroups, hence hold for $P_\g$ by Theorem~\ref{k4free}. 
The product $\underset{X \in \X(\g)}{\prod} P_X$ is a right-angled Artin group, so acts freely and properly on a CAT(0) cube complex, hence $P_\g$ does as well by Theorem~\ref{k4free}.
\end{proof}

The arrangement $\A_\g$ associated to a $K_4$-free graph \g\ is a decomposable arrangement, so by \cite[Prop.~4.5.6]{CFR10}, the residual nilpotence of $P_\g$ implies that every braid monodromy presentation of $P_\g$ is ``conjugation-free." 
See \cite{CFR10} and the references therein for details.

The existence of the geometric action of Corollary~\ref{props} implies $P_\g$ is a-T-menable, which has several strong consequences \cite{CCJJV01}.  
This action is not cocompact in general - see Corollary~\ref{kpi1}. 

The results above have a slight generalization, based on the analogue of Lemma~\ref{twoof} for 4-cliques.

\begin{lemma}\label{fiveof} Let $X$ be a 4-clique of \g\ and $g \in P_\g$. Suppose $[g,a_{ij}]=1$ for five of the six elements $\{i,j\}$ of $\binom{X}{2}$. Then $[[P_X,P_X],g]=1$ in $P_\g$.
\end{lemma}

\begin{proof} Let $z'$ be the image of the central element $\Delta^2 \in P_n$ under the quotient map $P_n \to P_X$. Then $P_X$ is generated by $z'$, which is central, and any five of the six elements $a_{ij}$, $\{i,j\} \in \binom{X}{2}$. Letting $H$ denote the subgroup generated by those five elements, one has $[P_X,P_X]=[H,H]$. The result follows as in the proof of Lemma \ref{twoof}.
\end{proof}

\begin{prop}\label{4clique} Suppose there are no 4-cliques in \g\ that intersect in a 3-clique of \g. Let $X$ be a 4-clique of \g\ and $\{r,s\} \not \subseteq X$. Then 
\[
[[P_X,P_X],a_{rs}]=1
\]
 in $P_\g$.
\end{prop}

\begin{proof} Because $X$ is clique and $\{r,s\} \not \subseteq X$, the hypothesis implies there is at most one $t \in X$ such that $\{r,s,t\}$ is a 3-clique. Then $\{a_{ij},a_{rs}\}$ is transverse to $\X(\g)$ for at least five of the six elements $\{i,j\}$ of $\binom{X}{2}$. The claim then follows from Proposition~\ref{comm} and Lemma~\ref{fiveof}.
\end{proof}

We will refer to the hypothesis of Proposition~\ref{4clique} by saying that the 4-cliques of \g\ are {\em almost disjoint}. This property holds when the submatroids of the cycle matroid of \g\ corresponding to the 4-cliques meet in at most one point.

\begin{thm} \label{fourinj} Suppose the 4-cliques of \g\ are almost disjoint. Then
\[
\rho_{\X(\g)} \colon P_\g \to \underset{X \in \X(\g)}{\displaystyle{\prod}} P_X
\]
is injective.
\end{thm}

\begin{proof} This now follows from Theorem~\ref{inj} using Propositions~\ref{modular}, \ref{comm}, \ref{triple}, and \ref{4clique}.
\end{proof}

\begin{cor} Suppose the 4-cliques of \g\ are almost disjoint. Then $P_\g$ is torsion-free, residually torsion-free nilpotent, linear, and residually finite. If \g\ is not $K_4$-free then $P_\g$ is not residually free. Moreover, $P_\g$ acts freely and properly on a CAT(0) complex.
\end{cor}

\begin{proof} Since $P_2\cong \Z$, $P_3\cong F_2 \times \Z$, and $P_4$ have the listed properties, which are preserved under direct products and inherited by subgroups, the first statement follows from Theorem~\ref{fourinj}. Since $P_4$ is not residually free \cite{CFR11}, the second statement follows from Proposition~\ref{modular}. By \cite{BraMcC2010}, $B_n$, $n=2,3,4$, acts freely and properly on a CAT(0) complex, hence $P_\g$ does as well.
\end{proof}

In this setting the geometric action may or may not be cocompact. It is not known whether $P_4$ is an a-T-menable group.

The hypothesis of Theorem~\ref{fourinj} is not strictly necessary, as we have already seen in Example~\ref{ex:deletedK5}. As in this example, if \g\ is a graph on at most $|V|=5$ vertices, then either \g\ is $K_4$-free or \g\ has a $(|V|-1)$-clique. In either instance, the map $\rho_{\X(\g)}$ is injective. 

\begin{ex} \label{ex:NOTinj}
For the graph $\g$ on $V=[6]$ with $E_\g=\binom{V}{2} \setminus \left\{\{4,5\},\{1,6\}\right\}$, the family of maximal cliques is
\[
\X(\g)=\left\{\{1,2,3,4\},\{1,2,3,5\},\{2,3,4,6\},\{2,3,5,6\}\right\};
\] 
see the accompanying figure. The graph $\g$ has 4-cliques meeting in 3-cliques. Also, the induced subgraph $\g_{1456}$ is a 4-cycle, so \g\ is not chordal. 
Nevertheless, the map $\rho_{\X(\g)}$ is still injective, as we now sketch.

\begin{figure}\label{fig:bipyramid}
\centering
\includegraphics[scale=.25]{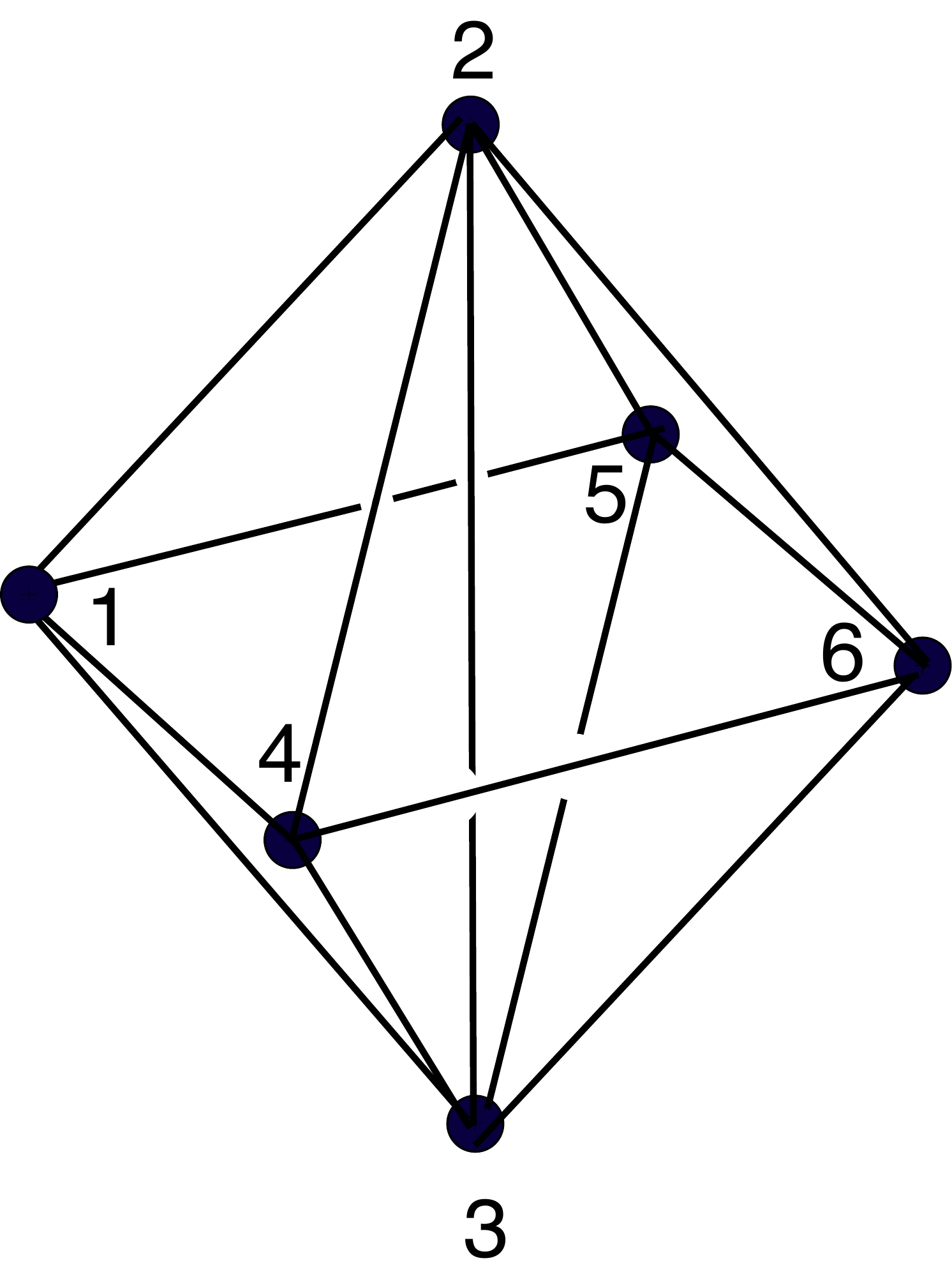} \hskip .5 truein
\includegraphics[scale=.23]{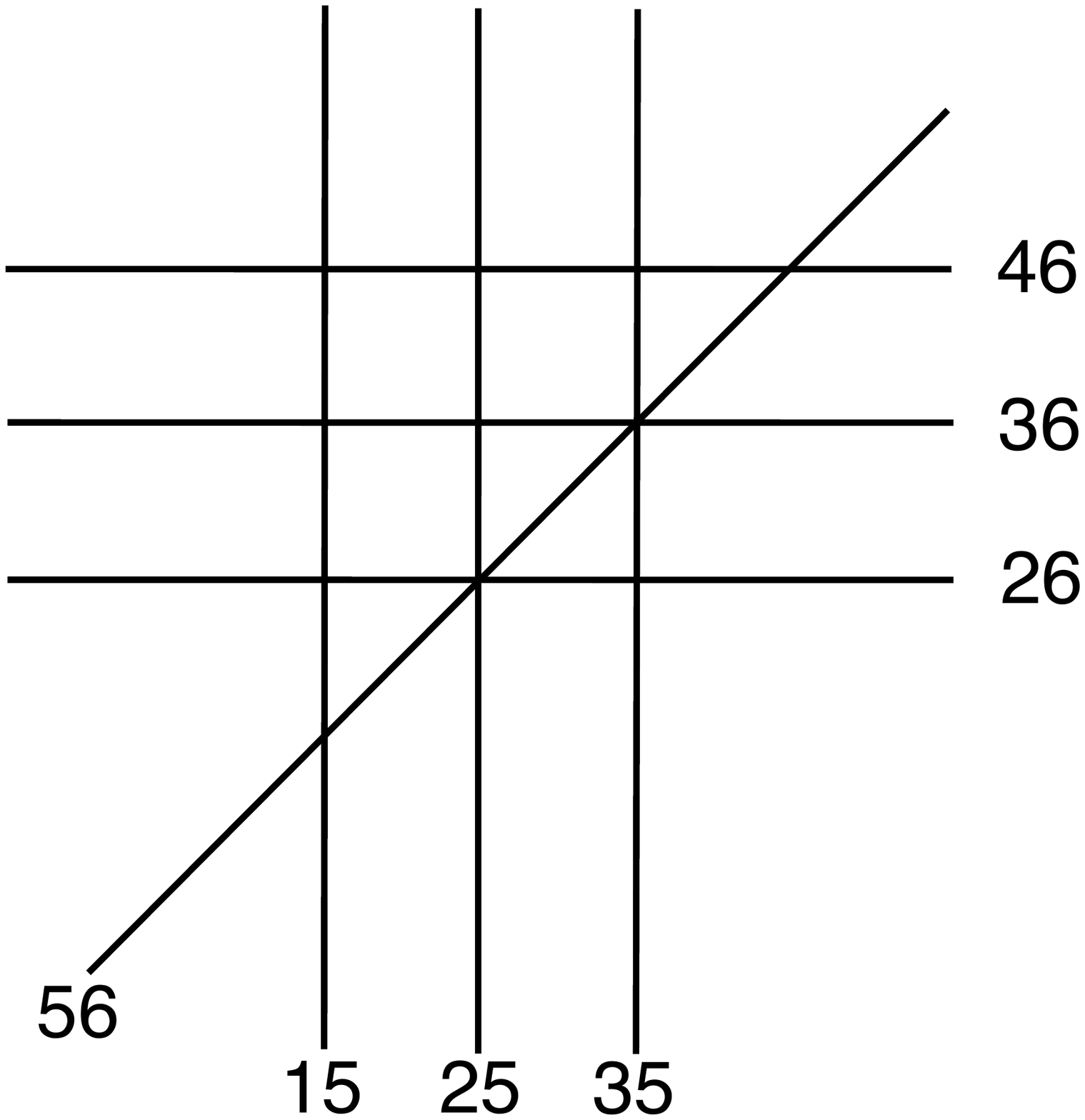}
\caption{The graph \g\ and arrangement $\A^{\g}_{6,4}$ of Example~\ref{ex:NOTinj}.}
\end{figure}

The homomorphism  $\rho_{\X(\g)}$ is given by
\[
\rho_{\X(\g)} \colon P_\g \longrightarrow P_{1234} \times P_{1235} \times P_{2346} \times P_{2356} \cong P_4 \times P_4 \times P_4 \times P_4.
\]
Let $P^\g_{6,4}=P_{6,4}/\langle\langle a_{45},a_{16}\rangle\rangle$ be the group of the graphic discriminantal arrangement $\A^\g_{6,4}$ associated to  
\g\ and the clique $[4]$, 
see Remark \ref{rem:deleted disc}. 
The map $\rho_\X$ is injective if and only if the restriction $\psi=\rho_\X|_{P^\g_{6,4}}$ is. This restriction has image  contained in the subgroup $F_3 \times F_3 \times \bar{P}_{2356}$ of $P_{1235}\times P_{2346} \times P_{2356}$, where $F_3<P_{1235}$ is generated by $\{a_{15},a_{25},a_{35}\}$, $F_3<P_{2346}$ is generated by $\{a_{26},a_{36},a_{46}\}$, and $\bar{P}_{2356}=\ker(P_{2356} \to P_{23})$.

The group $P^\g_{6,4}$ has generating set $Y=\{a_{ij}\mid i<j, j\in\{5,6\}, ij\neq 45,16\}$, and relations given by the pure braid relations in $P_{2356}$ (not involving $a_{23}$), along with (after simplification using Remark \ref{rem:equiv pure braid rels}) $[a_{15},a_{i6}]=1$, $2\le i \le 6$, $[a_{i5},a_{46}]=1$, $1\le i \le 3$, and $[a_{46},a_{56}]=1$. By Proposition \ref{conjfree}, the family 
\[
\X=\{S_1,S_2,S_3\}=\{\{a_{15},a_{25},a_{35}\},\{a_{26},a_{36},a_{46}\},\{a_{25},a_{35},a_{26},a_{35},a_{56}\}\}
\] 
of subsets of $Y$ is retractive. 

Theorem \ref{inj} does not inform on the kernel of $\psi=\rho_\X$. Using the relations in $P^\g_{6,4}$, one can check that $[[a_{36}^{}a_{35}^{},a_{36}^{}a_{25}^{}],a_{15}^{}]= [[a_{35}^{},a_{25}^{}],a_{15}^{}]$. This last expression is nontrivial in the (free) subgroup of $G=P^\g_{6,4}$ generated by $\{a_{15},a_{25},a_{35}\}$. Hence, $[[G_{S_3},G_{S_3}],a_{15}]\neq 1$, and condition Theorem \ref{inj} (ii) fails. We instead utilize the results (and terminology) of \cite[Section 3]{CFR10}, in particular, \cite[Thm.~3.2.13]{CFR10}, which shows that the kernel of $\psi = \rho_\X$ is generated by monic commutators in $P^\g_{6,4}$ transverse to $\X$.

Let $w=\beta^p(y_1^{\epsilon_1},\dots,y_p^{\epsilon_p})$ be such a monic commutator, a weight $p$ bracket arrangement, where $y_i\in Y$ is a generator of $P^\g_{6,4}$ and $\epsilon_i=\pm 1$ for each $i$, see \cite{MKS76,CFR10} for details. Since $w$ is transverse to $\X$, $w$ must involve at least one generator of $P^\g_{6,4}$ not in $S_3$. Without loss, assume $w$ involves $a_{15}$. Then, $w$ must also involve one of the generators $a_{i6}$, $2\le i\le 5$, not in $S_1$. We show that $w=1$ in $P^\g_{6,4}$ by induction on the bracket weight $p\ge 2$.

In the base case $p=2$, up to conjugation and taking inverses, $w=[a_{i6},a_{15}]=1$, since $a_{15}$ commutes with $a_{i6}$ in $P^\g_{6,4}$ as noted above. For general $p \ge 2$, write
$w=\beta^p(x_1, \ldots, x_p)=[\beta^k(x_1 \ldots, x_k),\beta^\ell(x_{k+1}, \ldots, x_p)]$, where $x_i=y_i^{\epsilon_i}$ and $p=k+\ell$ with $k,\ell\ge 1$. If $\{a_{15},a_{i6}\} \subseteq \{y_1,\dots,y_k\}$, then $\beta^k(x_1 \ldots, x_k)=1$ by induction, and $w=1$. Similarly, $w=1$ if $\{a_{15},a_{i6}\} \subseteq \{y_{k+1},\dots,y_p\}$. 

Without loss, assume that $a_{i6} \in \{y_1,\dots,y_k\}$ and $a_{15} \in \{y_{k+1},\dots,y_p\}$. The subgroup $N$ generated by $\{a_{26},a_{36},a_{46},a_{56}\}$ is normal in $P^\g_{6,4}$, so $[P_{6,4}^\g,N]\subseteq N$. Since $x_i \in N$, it follows that $\beta^{k}(x_1, \ldots, x_{k}) \in N$. Since $a_{15}$ commutes with all the generators of $N$ in $P^\g_{6,4}$, we have $w=[\beta^{k}(x_1, \ldots, x_{k}),a_{15}]=1$. If $\ell>1$, write $\beta^\ell(x_{k+1}, \ldots, x_p)=[u,v]$, where $u$ and $v$ are monic commutators of weight at least one. Assume $u$ involves $a_{15}$ (from which the case $v$ involves $a_{15}$ follows easily). Write $\beta=\beta^{k}(x_1, \ldots, x_{k})$. Inducting on $\ell$ and noting that $[\beta,v^{-1}]\in N$, we have $[\beta,u]=1$ and $[[v,\beta],u^v]=[[\beta,v^{-1}],u]^v=1$ by the inductive hypothesis. 
Then the Hall-Witt identity~\cite[Thm.~5.1\,(11)]{MKS76} yields
\[
w^{-1}=[[u,v],\beta]=[[u,v],\beta^u]=[u^v,[v,\beta]]\cdot [v^\beta,[\beta,u]]=1.
\] 
This completes the induction on $\ell$, and hence the induction on $p$. 
It follows that the maps $\psi$ and $\rho_{\X(\g)}$ are injective.
\end{ex}

\begin{rem} 
\label{rem:seven}
The homomorphism $\rho_{\X(\g)}$ is, in fact, injective for any graph on fewer than 7 vertices. As noted previously, this is the case for graphs on at most 5 vertices. For the complete graph $K_n$, the map $\rho_{\X(K_n)}$ is simply the identity map of the pure braid group $P_n$. If $\g'$ is obtained from $K_6$ by either deleting a single edge, or by deleting two adjacent edges, then $\g'$ satisfies the hypotheses of Corollary \ref{cor:neighbor},
which may be used to show that $\rho_{\X(\g')}$ is injective. Deleting two nonadjacent edges from $K_6$ results in a graph isomorphic the graph \g\ of Example \ref{ex:NOTinj}.

It remains to consider subgraphs obtained from \g\ by deleting edges. 
Deleting the edge $\{2,3\}$ results in a $K_4$-free graph $\g'$, and $\rho_{\X(\g')}$ is injective by Theorem \ref{k4free}. Deleting any one of the edges in the 4-cycle subgraph $\g_{1456}$ results in a graph $\g'$ satisfying the hypotheses of Corollary \ref{cor:neighbor}, 
which may be used to show that $\rho_{\X(\g')}$ is injective. Deleting one of the remaining edges, say $\{2,6\}$, results in a graph $\g'$ for which the associated graphic discriminantal arrangement $\A^{\g'}_{6,4}$ is a decone of the arrangement $\A(3,2,2)$ considered in \cite[Section 4.4]{CFR10}. The injectivity of the restriction of $\rho_{\X(\g')}$ to the graphic 
discriminantal arrangement group 
$P^{\g'}_{6,4}=P_{6,4}/\langle\langle a_{45},a_{16},a_{26}\rangle\rangle$ is established in \cite[Thm.~4..4.1]{CFR10}. As above, it follows that $\rho_{\X(\g')}$ is injective.

Deleting 2 or more edges from \g\ results in graphs $\g'$ which are either covered by the above considerations, or have almost disjoint 4-cliques. In the last instance, $\rho_{\X(\g')}$ is injective by Theorem \ref{fourinj}.
\end{rem}

Example \ref{ex:NOTinj} and Remark \ref{rem:seven} 
lead us to propose the following problems:

\begin{problem} \ 
\begin{enumerate}
\item[(a)] Characterize the graphs \g\ for which $\rho_{\X(\g)}$ is injective.

\smallskip

\item[(b)] Characterize the graphs \g\ for which $P_\g$ embeds in a product of pure braid groups.
\end{enumerate}
\end{problem}

\section{Homological finiteness type}
\label{finiteness}
We return to the general setting of Section~\ref{retractive} to recall some further results from \cite[Section 4]{CFR10}. Let $G$ be a group with finite generating set $Y=\{y_1, \ldots, y_n\}$. Assume that $z=y_1 y_2\cdots y_n$ is central in $G$, and let $\bar{G}=G/\langle z \rangle$. Let $\X \subseteq 2^Y$, and for $S \in \X$ let $\rho_S \colon G \to G_S$ be the canonical quotient map and $\rho_\X = \prod \rho_S \colon G \to \underset{S \in \X}{\displaystyle{\prod}} G_S$ as in Section~\ref{retractive}.

Then $\rho_S(z)$ is central in $G_S$; let $\bar{G}_S=G_S/\langle \rho_S(z) \rangle$. Then there is a homomorphism $\bar{\rho}_\X \colon \bar{G} \to \underset{S \in \X}{\displaystyle{\prod}} \bar{G}_S$ and a commutative diagram
\[
\xymatrix{
G \ar[r]^-{\rho_\X} \ar[d] & \underset{S \in \X}{\displaystyle{\prod}} G_S \ar[d]\\
\bar{G} \ar[r]^-{\bar{\rho}_\X} &  \underset{S \in \X}{\displaystyle{\prod}} \bar{G}_S.
}
\]

Let $\L_\X$ be the {\em incidence graph} of $Y$ with \X, that is, $\L_\X$ is the bipartite graph with vertex set $Y \cup \X$ and edges $\{y,S\}$ for $y \in S$. In the statements below, $\L_\X$ is to be considered as a 1-dimensional cell complex. The next three propositions are proved in \cite{CFR10}.

\begin{prop}\label{proj} Suppose $\L_\X$ is connected. Then the kernel of $\rho_\X$ projects isomorphically onto the kernel of $\bar{\rho}_\X$. In particular, if $\rho_\X$ is injective, then $\bar{\rho}_\X$ is injective.
\end{prop}

\begin{prop} \label{norm} Suppose $|S \cap S'|\leq 1$ for all $S\neq S'$ in $\X$. Then the image of $\bar{\rho}_\X$ is normal in $\underset{X \in \X}{\displaystyle{\prod}}\bar{G}_S$, and the quotient $\left(\underset{S \in \X}{\displaystyle{\prod}} \bar{G}_S\right)\Big{\slash}\bar{\rho}_\X(\bar{G})$ is isomorphic to $H^1(\L_\X, \Z)$. In particular, $\left(\underset{S \in \X}{\displaystyle{\prod}} \bar{G}_S\right)\Big{\slash}\bar{\rho}_\X(\bar{G})$ is a free abelian group of rank equal to the first betti number of $\L_\X$.
\end{prop}

Now suppose $\bar{G}_S$ is a free group for every $S \in \X$. Then one can apply the results of \cite{MMV01} to identify the homological finiteness type of $\bar{\rho}_\X(\bar{G})$. Refer to \cite{CFR10} for the details of the computation. Recall, a group $G$ is of type $FP_m$ if there is a partial resolution of $\Z$ as a trivial $\Z[G]$-module of finite type and length at most $m$. If $G$ is finitely-presented, then $G$ is of type $FP_m$ if and only if there is a $K(G,1)$ space with finite $m$-skeleton, in which case $G$ is said to have type $F_m$. Our groups are all finitely-presented, but to avoid a conflict of notation in the sequel, we will use the former notion. 

\begin{prop}[\cite{CFR10}] \label{type} Suppose $\L_\X$ has no isthmuses and contains a cycle. Then $\bar{\rho}(\bar{G})$ is of type $FP_{m-1}$ and not of type $FP_m$, where $m=|\X|$.
\end{prop}

Let \g\ be a graph with vertex set $V=[n]$ and edge set $E_\g \subseteq \binom{V}{2}$. We apply the results stated above with $G=P_\g$, $Y=\{a_{ij} \mid \{i,j\} \in E_\g\}$, and  $\X=\X(\g) \subseteq 2^Y$ corresponding to the set of maximal cliques of \g. 

Let $z \in P_\g$ be the image of the full twist $\Delta^2$ under the quotient map $P_n \to  P_\g$. Then $z$ is central in $P_\g$ and $z=y_1 y_2\cdots y_m$, where $Y=\{y_1, \ldots, y_m\}$ is labeled compatibly with the factorization of $\Delta^2$ from Section~\ref{purebraid}. The group $\bar{P}_\g=P_\g/\langle z \rangle$ is isomorphic to the fundamental group of the complement of the projectivization of the graphic arrangement $\A_\g$.

\begin{thm}\label{graphtype} Suppose \g\ is a connected graph in which every maximal clique has cardinality three. Let $m$ be the number of 3-cliques in \g, and let $\L_\g$ be the incidence graph of edges and 3-cliques in \g.

\begin{enumerate}
\item If $b_1(\L_\g)=0$ then $\bar{P}_\g$ is isomorphic to $(F_2)^m$.
\item If $b_1(\L_\g)>0$ and $\L(\g)$ has no isthmuses, then $\bar{P}_\g$ is of type $FP_{m-1}$ and not of type $FP_m$.
\end{enumerate}
\end{thm}

\begin{proof} The hypotheses imply that \g\ is a $K_4$-free graph, that $\X=\X(\g)$ consists of 3-cliques, and that $\L_\g=\L_\X$. Since \g\ is connected and has no maximal 2-cliques, $\L_\X$ is connected. By the pure braid relations, the projectivized groups $\bar{P}_X$, $X \in \X$, are isomorphic to $F_2$, the rank-two free group. Then by Theorem~\ref{k4free} and Proposition~\ref{proj}, the homomorphism
\[
\bar{\rho}_\X \colon \bar{P}_\g \to \underset{X \in \X}{\displaystyle{\prod}} \bar{P}_X
\]
embeds $\bar{P}_\g$ into the product $(F_2)^m$ of free groups, where $m=|\X|$.

Since two 3-cliques share at most one edge, it is the case that $|S \cap S'|\leq 1$ for all $S, S' \in \X$, $S \neq S'$. Statements (i) and (ii) then follow from Propositions~\ref{norm} and \ref{type}. 
\end{proof}

The first conclusion above can be deduced from the main result of \cite{Fan97}, see also \cite[Section 5]{DenGarToh14}. 

The projection $P_\g \to \bar{P}_\g$ splits, so $P_\g \cong \bar{P}_\g \times \Z$,  yielding similar statements for $P_\g$.

\begin{cor} \label{types}  Let \g\ be a connected graph in which every maximal clique has cardinality 3. Let $m$ be the number of 3-cliques in \g, and let $\L_\g$ be the incidence graph of edges and 3-cliques in \g.
\begin{enumerate}
\item If $b_1(\L_\g)=0$ then $P_\g$ is isomorphic to $(F_2)^m \times \Z^m$.
\item If $b_1(\L_\g)>0$ and $\L_\g$ has no isthmuses, then $P_\g$ is of type $FP_{m-1}$ and not of type $FP_m$.
\end{enumerate}
\end{cor}

\begin{proof} The groups $P_X$, $X \in \L_\g$, are isomorphic to $F_2 \times \Z$, yielding the first statement. The second follows from \cite{MMV01}, with the observation that the clique complexes of the graphs associated to the right-angled Artin groups $(F_2)^m$ and $(F_2)^m \times \Z^m$ have the same connectivity.
\end{proof}

Recall that a hyperplane arrangement \A\ in $\C^\ell$ is called a {\em $K(\pi,1)$ arrangement} if its complement $U_\A=\C^\ell \setminus \bigcup_{H \in \A} H$ is an aspherical space.

\begin{cor}\label{kpi1} Suppose \g\ is a $K_4$-free graph. Let $\L_\g$ be the incidence graph of edges and 3-cliques in \g. Let $\A_\g$ be the graphic arrangement associated with \g. If $\L_\g$ contains a cycle, then $\A_\g$ is not a $K(\pi,1)$ arrangement.
\end{cor}

\begin{proof} By deleting edges from \g, one can obtain an induced subgraph $\g'$ of \g\ satisfying the hypothesis of Theorem~\ref{graphtype}. The corresponding graphic arrangement $\A_{\g'}$ is a localization of $\A_\g$, and is not a $K(\pi,1)$ arrangement because $P_{\g'}=\pi_1(U_{\g'})$ is not of finite type, while the arrangement complement $U_{\g'}$ has the homotopy type of a finite complex. It follows that $\A_\g$ is not $K(\pi,1)$ - see \cite{Par00,Fa95}.
\end{proof}

\section{Graphic braid groups}
\label{full}
The full and pure braid groups are related by the short exact sequence
\[
1 \to P_n \to B_n \overset{p}{\to} S_n \to 1.
\]
Proposition~\ref{conj} yields the following description of the conjugation action of $B_n$ on $P_n$.

\begin{prop}\label{normalize} Let $\sigma \in B_n$ and let $\s=p(\sigma) \in S_n$. Let $1 \leq i <j \leq n$, and let $\{\s(i),\s(j)\}=\{r,s\}$ with $1 \leq r < s \leq n$. Then $\sigma a_{ij} \sigma^{-1}$ is a conjugate of $a_{rs}$ by an element of $P_n$.
\end{prop}

\begin{proof} The statement holds for the generators $\sigma_i$ of $B_n$, by Proposition~\ref{conj}, and that implies the result since the conjugation action is homomorphic.
\end{proof}

Let \g\ be a simple graph on vertex set $V=[n]$ and edge set $E_\g \subseteq 2^V$  as before. The automorphism group of \g\ is the subgroup $\Aut(\g)$ of $S_n$ consisting of those elements $\s$ for which $\{\s(i),\s(j)\} \in E_\g$ for all $\{i,j\} \in E_\g$. Recall the graphic pure braid group $P_\g$ is the quotient of $P_n$ by the normal subgroup $N$ normally generated by $\{a_{ij} \mid \{i,j\}  \in \bar{E}_\g\}$, where $\bar{E}_\g$ is the complement of $E_\g$ in $\binom{V}{2}$. Let $\tilde{B}_\g = p^{-1}(\Aut(\g)) \subseteq B_n$. Then $P_n \subseteq \tilde{B}_\g$.

\begin{cor} \label{normal} The normal subgroup $N$ of $P_n$ is normal in $\tilde{B}_\g$.
\end{cor}

\begin{proof} Let $\sigma \in \tilde{B}_\g$ and $1 \leq i<j \leq n$ with $\{i,j\} \in \bar{E}_\g$. Let $\s=p(\sigma)$. Then $\s \in \Aut(\g)$, so $\{\s(i),\s(j)\} \in \bar{E}_\g$. The statement then follows from Proposition~\ref{normalize}.
\end{proof}

\begin{dfn} The graphic braid group $B_\g$ associated with \g\ is the quotient $\tilde{B}_\g/N$ of $\tilde{B}_\g=p^{-1}\left(\Aut(\g)\right)$ by the normal subgroup $N$ normally generated by the pure braids $a_{ij}$, $\{i,j\}  \in \bar{E}_\g$.
\end{dfn}

\begin{thm} There is a short exact sequence
\[
1 \to P_\g \to B_\g \to \Aut(\g) \to 1.
\]
\end{thm}

\begin{proof} This follows immediately from Corollary~\ref{normal}.
\end{proof}

\begin{rem} The group $\Aut(\g)$ acts on the graphic arrangement complement $U_\g$. The fundamental group of the space of orbits $\CC_n^\g$ can be interpreted as the group of braids that permute their endpoints by automorphisms of \g, modulo level-preserving isotopy, with the strands corresponding to non-edges of \g\ allowed to cross. The action of $\Aut(\g)$ on $U_\g$ is not free in general, but one can interpret $B_\g$ as the orbifold fundamental group of the orbit space $\CC_n^\g$, 
giving, as in \cite{Looi08}, a topological interpretation of the short exact sequence above:
\[
1 \to \pi_1(U_\g) \to \pi_1^{\rm orb}(\CC_n^\g) \to \Aut(\g) \to 1.
\]
\end{rem}

\begin{ex} Let \g\ be the graph with vertex set $V=[3]$ and edge set $E_\g=\left\{\{1,2\},\{1,3\}\right\}$. The graphic arrangement complement $U_\g$ is $\{(x,y,z) \in \C^3 \mid x \neq z, y \neq z\}$, homeomorphic to $(\C^\times)^2 \times \C$. The automorphism group $\Aut(\g)$ is generated by the transposition $\s_2=(23)$; the points of $U_\g$ on the plane $y=z$ are fixed by $\Aut(\g)$.

The group $\tilde{B}_\g$ is generated by $\{a_{12},a_{13},\sigma_2\}$, and the graphic braid group $B_\g$ is the quotient of this group by the normal subgroup $N$ of $P_3$ generated by $a_{23}=\sigma_2^2$. 
By Propositions~\ref{conj} and \ref{comm}, $\sigma_2 a_{12} \sigma_2^{-1}=a_{13}$ and $\sigma_2 a_{13} \sigma_2^{-1} = a_{12}$, and $[a_{12},a_{13}]=1$, modulo $N$. Thus $\sigma_2$ represents an element of order two in $B_\g$. One easily shows $B_\g$ is isomorphic to $\Z^2 \rtimes \Z_2$. Then $B_\g$ acts by isometries on the euclidean plane, faithfully, properly, and cocompactly, with finite stabilizers, as the group of symmetries of the tiling of the plane by isosceles right triangles.
\end{ex}

The preceding example shows that $B_\g$ need not be torsion-free. The same phenomenon occurs for any graph \g\ having an automorphism that interchanges two non-adjacent vertices. Higher-order torsion arises in the same way in many examples: for instance, if \g\ is the $n$-wheel then $B_\g$ has elements of order $n$. Thus $B_\g$ is not residually free in general, and may have infinite cohomological dimension; it is not clear if the latter can be the case for $P_\g$. The results of the preceding section do apply to $B_\g$.

\begin{thm} Let \g\ be a connected graph with every maximal clique of cardinality 3. Let $m$ be the number of 3-cliques in \g, and let $\L_\g$ be the incidence graph of edges and 3-cliques in \g. Assume $\L(\g)$ has no isolated vertices. If $b_1(\L)>0$ then $B_\g$ is of type $FP_{m-1}$ and not of type $FP_m$.
\end{thm}

\begin{proof} Since $P_\g$ is a subgroup of finite index in $B_\g$, the statement follows from Corollary~\ref{types} and \cite[Prop. 5.1]{Br82}.
\end{proof}

\begin{ack} This research began with work by Anthony Caine and Daniel Malcolm, undergraduate research students of the second-named author at NAU in 2012 and 2015, respectively. We acknowledge their contributions. The manuscript was prepared while the second author was in residence at IMAG at University of Montpellier; he thanks the institute and Cl\'ement Dupont for hospitality, and Dupont and Thomas Haettel for helpful conversations.
\end{ack}

\providecommand{\bysame}{\leavevmode\hbox to3em{\hrulefill}\thinspace}


\begin{thebibliography}{10}

\bibitem{Bi75}
J.~Birman, \emph{Braids, links, and mapping class groups}, Annals of
  Mathematics Studies, no.~82, Princeton University Press, Princeton, N.J.,
  1975, \href{http://www.ams.org/mathscinet-getitem?mr=0375281}{MR0375281}.

\bibitem{BraMcC2010}
T.~Brady and J.~McCammond, \emph{Braids, posets and orthoschemes}, Algebr.
  Geom. Topol. \textbf{10} (2010), 2277--2314,
  \href{http://www.ams.org/mathscinet-getitem?mr=2745672}{MR2745672}.

\bibitem{Br82}
K.~Brown, \emph{Cohomology of groups}, Springer Verlag, Berlin Heidelberg New
  York, 1982,
  \href{http://www.ams.org/mathscinet-getitem?mr=0672956}{MR0672956}.

\bibitem{CCJJV01}
P.-A. Cherix, M.~Cowling, P.~Jolissaint, P.~Julg, and A.~Valette, \emph{Groups
  with the {H}aagerup property}, Progress in Mathematics, vol. 197,
  Birkh\"auser Verlag, Basel, 2001,
  \href{http://www.ams.org/mathscinet-getitem?mr=1852148}{MR1852148}.

\bibitem{C01}
D.~Cohen, \emph{Monodromy of fiber-type arrangements and orbit configuration
  spaces}, Forum Math. \textbf{13} (2001), 505--530,
  \href{http://www.ams.org/mathscinet-getitem?mr=1830245}{MR1830245}.

\bibitem{CFR10}
D.~Cohen, M.~Falk, and R.~Randell, \emph{Discriminantal bundles, arrangement
  groups, and subdirect products of free groups}, Eur. J. Math. (to appear),
  \href{https://arxiv.org/abs/1008.0417}{\tt arXiv:1008.0417}.

\bibitem{CFR11}
\bysame, \emph{Pure braid groups are not residually free}, CRM Series, 14,
  pp.~213--230, Ed. Norm., Pisa, 2012,
  \href{http://www.ams.org/mathscinet-getitem?mr=3203640}{MR3203640}.

\bibitem{CS97}
D.~Cohen and A.~Suciu, \emph{The braid monodromy of plane algebraic curves and
  hyperplane arrangements}, Comment. Math. Helv. \textbf{72} (1997), 285--315,
  \href{http://www.ams.org/mathscinet-getitem?mr=1470093}{MR1470093}.

\bibitem{DenGarToh14}
G.~Denham, M.~Garrousian, and \c{S}. Toh\v{a}neanu, \emph{Modular decomposition
  of the {O}rlik-{T}erao algebra}, Ann. Comb. \textbf{18} (2014), 289--312,
  \href{http://www.ams.org/mathscinet-getitem?mr=3206154}{MR3206154}.

\bibitem{DG81}
J.~Dyer and E.~Grossman, \emph{The automorphisms groups of the braid groups},
  Amer. Math. J. \textbf{103} (1981), 1151--1169,
  \href{http://www.ams.org/mathscinet-getitem?mr=0636956}{MR0636956}.

\bibitem{FN62}
E.~Fadell and L.~Neuwirth, \emph{Configuration spaces}, Math. Scand.
  \textbf{10} (1962), 111--118,
  \href{http://www.ams.org/mathscinet-getitem?mr=0141126}{MR0141126}.

\bibitem{Fa95}
M.~Falk, \emph{{$K(\pi,1)$} arrangements}, Topology \textbf{34} (1995),
  141--154, \href{http://www.ams.org/mathscinet-getitem?mr=1308492}{MR1308492}.

\bibitem{FaPr02}
M.~Falk and N.~Proudfoot, \emph{Parallel connections and bundles of
  arrangements}, Topology Appl. \textbf{118} (2002), 65--83,
  \href{http://www.ams.org/mathscinet-getitem?mr=1877716}{MR1877716}.

\bibitem{FR88}
M.~Falk and R.~Randell, \emph{Pure braid groups and products of free groups},
  Contemp. Math., vol.~78, pp.~217--228, Amer. Math. Soc., Providence, RI,
  1988, \href{http://www.ams.org/mathscinet-getitem?mr=0975081}{MR0975081}.

\bibitem{Fan97}
K.-M. Fan, \emph{Direct product of free groups as the fundamental group of the
  complement of a union of lines}, Michigan Math. J. \textbf{44} (1997),
  283--291, \href{http://www.ams.org/mathscinet-getitem?mr=1460414}{MR1460414}.

\bibitem{GL69}
E.~Gorin and V.~Lin, \emph{Algebraic equations with continuous coefficients and
  some problems of the algebraic theory of braids}, Math. USSR-Sb. \textbf{7}
  (1969), 569--596,
  \href{http://www.ams.org/mathscinet-getitem?mr=0251712}{MR0251712}.

\bibitem{Ku93}
J.~P.~S. Kung, \emph{Extremal matroid theory}, Contemp. Math., vol. 147,
  pp.~21--61, Amer. Math. Soc., Providence, RI, 1993,
  \href{http://www.ams.org/mathscinet-getitem?mr=1224696}{MR1224696}.

\bibitem{L-FS}
P.~Lima-Filho and H.~Schenck, \emph{Holonomy {$L$}ie algebras and the {$LCS$}
  formula for subarrangements of {$A_n$}}, Int. Math. Res. Not. IMRN 2009
  (2009), no.~8, 1421--1432,
  \href{http://www.ams.org/mathscinet-getitem?mr=2496769}{MR2496769}.

\bibitem{Looi08}
E.~Looijenga, \emph{Artin groups and the fundamental groups of some moduli
  spaces}, J. Topol. \textbf{1} (2008), 187--216,
  \href{http://www.ams.org/mathscinet-getitem?mr=2365657}{MR2365657}.

\bibitem{MKS76}
W.~Magnus, A.~Karrass, and D.~Solitar, \emph{Combinatorial group theory},
  second ed., Dover Publications Inc., 2004,
  \href{http://www.ams.org/mathscinet-getitem?mr=2109550}{MR2109550}.

\bibitem{Mal15}
D.~Malcolm, \emph{final report for {MAT} 485, {U}ndergraduate {R}esearch},
  Northern Arizona University, 2015.

\bibitem{MarMcC09}
D.~Margalit and J.~McCammond, \emph{Geometric presentations for the pure braid
  group}, J. Knot Theory Ramifications \textbf{18} (2009), 1--20,
  \href{http://www.ams.org/mathscinet-getitem?mr=2490001}{MR2490001}.

\bibitem{Mar12}
I.~Marin, \emph{Residual nilpotence for generalizations of pure braid groups},
  CRM Series, 14, pp.~389---401, Ed. Norm., Pisa, 2012,
  \href{http://www.ams.org/mathscinet-getitem?mr=3203649}{MR3203649}.

\bibitem{MMV01}
J.~Meier, H.~Meinert, and L.~VanWyk, \emph{On the {$\Sigma$}-invariants of
  {A}rtin groups}, Topology Appl. \textbf{110} (2001), 71--81,
  \href{http://www.ams.org/mathscinet-getitem?mr=1804699}{MR1804699}.

\bibitem{Par00}
L.~Paris, \emph{Intersection subgroups of complex hyperplane arrangements},
  Topology Appl. \textbf{105} (2000), 319--343,
  \href{http://www.ams.org/mathscinet-getitem?mr=1769026}{MR1769026}.

\bibitem{SV91}
V.~Schechtman and A.~Varchenko, \emph{Arrangements of hyperplanes and {L}ie
  algebra homology}, Invent. Math. \textbf{106} (1991), 139--194,
  \href{http://www.ams.org/mathscinet-getitem?mr=1123378}{MR1123378}.

\bibitem{Sta72}
R.~Stanley, \emph{Supersolvable lattices}, Algebra Universalis \textbf{2}
  (1972), 214--217,
  \href{http://www.ams.org/mathscinet-getitem?mr=0309815}{MR0309815}.

\bibitem{Zar17}
M.~Zaremsky, \emph{Separation in the {BNSR}-invariants of the pure braid
  groups}, Publ. Mat. \textbf{61} (2017), 337--362,
  \href{http://www.ams.org/mathscinet-getitem?mr=3677865}{MR3677865}.

\end{thebibliography}
\end{document}